\begin{filecontents*}{example.eps}
\documentclass[11pt]{article}
\usepackage{latexsym,amsmath,amsthm,verbatim,ifthen,amssymb}
\usepackage[all]{xy}

\newtheorem{theorem}{Theorem}[section]

\newtheorem{corollary}[theorem]{Corollary}
 \newtheorem{lemma}[theorem]{Lemma}
 \newtheorem{proposition}[theorem]{Proposition}
 \theoremstyle{definition}
 \newtheorem{definition}[theorem]{Definition}
 \theoremstyle{remark}
 \newtheorem{remark}[theorem]{Remark}
 \newtheorem{example}[theorem]{Example}
 \numberwithin{equation}{subsection}

\newcommand{\bz}{\mathbb Z}

\newcommand{\sgn }{\text{sgn}}
\newcommand{\maca}{\operatorname{{\bf MC}}}
\newcommand{\cinfinito}{\operatorname{{\cal C}^\infty}}

 \newcommand{\dalg}{d}

 \newcommand{\cu }{\mathbb{Q}}

 \newcommand{\alg }{\text{CDGA}\ }
 \newcommand{\lie }{\text{DGL}\ }
 \newcommand{\coa }{\text{CDGC}\ }

 \newcommand{\kerB }{B^\sharp_+}
 \newcommand{\kerb }{B_+}

 \newcommand{\der }{\mathcal{D}er }

 \newcommand{\pl }{1}
 \newcommand{\pc }{2}

 \newcommand{\cinf }{\mathcal {C}^\infty }
 \newcommand{\map}{\text{map}}
 \newcommand{\im}{\text{Im}\,}
 \newcommand{\calc}{{\cal C}}

  \newcommand{\Wh }{\textrm{Wh}_\cu\, }

  \newcommand{\lineal }{{{\phi^\sharp }}}
  \newcommand{\efe }{f}
  \newcommand{\nil }{\textrm{nil}_\cu \,}
  \newcommand{\nilsin }{\textrm{nil}\,}
  \newcommand{\mc}{z}
   \newcommand{\limite}{{\displaystyle \lim_{\leftarrow_n}}\,}
\begin{document}

\title{$L_\infty $ rational homotopy of mapping spaces }


\author{Urtzi Buijs\footnote{Partially supported by the {\em MICIN}  grant MTM2010-15831 and by the Junta de
Andaluc\'\i a grant FQM-213.}, Yves F\'elix\footnote{Partially supported by the {\em MICIN}  grant MTM2010-18089.}{ } and Aniceto Murillo\footnote{Partially supported by the {\em MICIN}  grant MTM2010-18089 and by the Junta de
Andaluc\'\i a grants FQM-213 and P07-FQM-2863. \vskip 1pt 2000 Mathematics Subject
Classification: 55P62, 54C35.\vskip
 1pt
 Key words and phrases: mapping space, rational homotopy theory, $L_\infty $ algebra, $L_\infty $ model of a space.}}



\maketitle

\begin{abstract}
In this paper we describe explicit $L_\infty$ algebras modeling the rational homotopy type of any component of the spaces
 $\map(X,Y)$ and $\map^*(X,Y)$ of free and pointed maps between the finite nilpotent CW-complex $X$ and the finite type nilpotent CW-complex $Y$. When $X$ is of finite type, non necessarily finite, we also show that the algebraic covers of these $L_\infty$ algebras model the universal covers of the corresponding mapping spaces. 
\end{abstract}

\section{Introduction}\label{intro}
Unless explicitly stated otherwise, all spaces considered in this paper shall be nilpotent so that
 they admit classical localization, and any component of mapping spaces between them remains nilpotent.
Let $f\colon X\to Y$ be a continuous based map between path connected,  finite type, CW-complexes
 and denote by  $\map_f(X,Y)$ (resp. $\map_f^*(X,Y)$)  the path component of $f$ in the space of
  continuous maps (resp. pointed continuous maps) from $X$ to $Y$. From the rational point of view, and
  whenever $X$ is finite, the homotopy behavior of these complexes, often infinite but always of finite type, was first described by Haefliger in the fundamental reference \cite{Ha}. Since then, the rational homotopy of mapping spaces has been
   extensively studied in the two classical settings, the Sullivan and the Quillen  approach, each of which is based, roughly speaking, in the replacement of the classical homoto-\break py category by that of commutative differential graded algebras, \alg's henceforth, and differential graded Lie algebras, \lie's henceforth, respectively.

\medskip

The version up to homotopy of the latter algebraic structure derives in the concept of $L_\infty$ algebra, introduced
first in the context of deformation theory of algebraic structures \cite{ss} and highly used
 since  in quite different geometrical settings (see \cite{cs} or \cite{kon} for instance). As these  objects are also susceptible of
  being geometrically realized \cite{getzler,hen} it is natural to search for $L_\infty$ algebras modeling the rational homotopy of
   mapping spaces as above. In the first attempt towards this end \cite{urtzi4}, a somewhat natural \lie structure in the space of derivations
   between  Quillen models of $X$ and $Y$ is deformed, via the now classical \emph{homotopy transfer theorem} \cite{fuk,kontsoib1,kontsoib2,bruno,merk} (which goes back to \cite{hub-kai,kai} for the case of  $A_{\infty}$ structures)
 to obtain an  $L_\infty$ algebra  modeling the based mapping space.  Besides restricting to  pointed maps,  hard computations often prevent this $L_\infty$ model to be used for practical purposes.

\medskip

We overcome these obstacles and give a complete answer by presenting  explicit $L_\infty$
models (see next section for precise definitions), of any component of both free and pointed mapping spaces.
   One of the most important features of the following is that it applies to the general situation (i.e., not only when $X$ is finite)
    in which case the components of the free and/or based mapping spaces are no longer of finite type. In what follows and along the paper, all considered  graded vector spaces, algebras of any kind and coalgebras are defined over a field of characteristic zero, which shall be $\mathbb Q$ when these algebraic objects are modeling the rational homotopy type of a given space.

\medskip

 Consider $\phi\colon C\to\mathcal{C}_\infty(L)$ a cocommutative differential graded coalgebra, \coa henceforth, model of $f\colon X\to Y$, in which $C$ is a \coa model of $X$ and $\mathcal{C}_\infty(L)$ is the Cartan-Chevalley-Eilenberg complex of an $L_\infty$ model of $Y$  (see next section for a brief compendium  on $L_\infty$ algebras from a rational homotopy point of view).  Then, the graded vector space $Hom(C,L)$ has a well known and natural structure of $L_\infty$ algebra with the usual differential and higher brackets given by
$$ \ell_k(f_1,\dots,f_k)=\left[f_1,\dots,f_k\right]=\ell_k \circ (f_1\otimes
              \cdots \otimes f_k) \circ \Delta^{(k-1)}
            $$
in which $k\ge 2$ and, in the last term, $\ell_k$ is the $k$-bracket on $L$. Moreover, for this structure, the induced map denoted also by $\phi\in Hom_{-1}(C,L)$ is a  Maurer-Cartan element which, in a classical and natural way, produces a perturbation $\ell^\phi_k$ of $\ell_k$ for which $Hom(C,L)$ is still an $L_\infty$ algebra,
$$\ell_k^\phi(f_1,\dots,f_k)= \sum_{j =0}^\infty \frac{1}{j!}[ \underbrace{\phi
,\dots ,\phi}_{j \text{ times}},  f_1,\dots ,f_k]\,.$$

Then, restricting to the non-negative part,
$$\mathcal{H}om(C,L)= \left\lbrace
           \begin{array}{c l}
              Hom_i(C,L)& \text{for $i\ge 1$},\\
              \mathcal{Z}(Hom_0(C ,L)) & \text{for $i=0$},
           \end{array}
         \right.
         $$
($\mathcal{Z}$ stands for cycles)  and whenever $Y$ is a finite type complex,  we prove:

\begin{theorem} \label{principallibre} If $X$ is a nilpotent finite complex, then
 $(\mathcal{H}om(C,L), \{\ell^\phi_k\}_{k\ge 1} )$  is
 an $L_\infty $
 model of {\em $\map_f(X,Y)$}.
 \end{theorem}
Replacing in the above procedure $C$ and $\Delta$ by $\overline C$ (the kernel of the augmentation) and  $\overline\Delta$ (the reduced  diagonal) we obtain:
\begin{theorem} \label{principalbasado} If $X$ is a nilpotent finite complex, then
 $(\mathcal{H}om(\overline C,L), \{\ell^\phi_k\}_{k\ge 1} )$  is
 an $L_\infty $
 model of {\em $\map_f^*(X,Y)$}.
 \end{theorem}

When $L$ is a DGL, the above process gives to $\mathcal{H}om(C,L)$ a DGL structure and we recover the results of \cite{urtzi3}. But in some situations, it is more easy to construct a $L_\infty$-model of $Y$ than a Lie model, and then our theorems apply.

\vspace{2mm}
On the other hand, if we assume $X$ to be of finite type (not necessarily finite) we also find $L_\infty$ structures modeling the corresponding component of the mapping space, even though, these are no longer of finite type. Denote by $\widetilde{\mathcal{H}om}(C,L)$ the ``universal" cover of $\mathcal{H}om(C,L)$, i.e.,
 $$\widetilde{\mathcal{H}om}(C,L)= \left\lbrace
           \begin{array}{c l}
              Hom_i(C,L)& \text{for $i\ge 2$},\\
              \mathcal{Z}(Hom_1(C ,L)) & \text{for $i=1$},
           \end{array}
         \right.
         $$
         Then, with the same $L_\infty$ structures as above:

         \begin{theorem} \label{casonofinito} If $X$ is a nilpotent finite type complex, then
 $\widetilde{\mathcal{H}om}(C,L)$  and $\widetilde{\mathcal{H}om}(\overline C,L)$ are
  $L_\infty $
 model of the universal covers of  $\widetilde{\map}_f(X,Y)$ and $\widetilde{\map}_f^*(X,Y)$ respectively.
 \end{theorem}

Considering that {\em nilpotent $L_\infty$ algebras} play a key role in the understanding of the geometric behavior of these algebraic structures \cite{getzler}, we then bound the {\em nilpotency order} (see \S4 for precise definitions) of the models above as follows:

\begin{theorem}\label{lemanilpotencia}
 $$\nilsin (\mathcal{H}om(C,L), \{\ell^\phi_k\}_{k\ge 1} )\le \nilsin (\mathcal{H}om(C,L), \{\ell_k\}_{k\ge 1} )\le \nilsin L.$$
 The same holds in the ``pointed case", i.e.,
 $$\nilsin (\mathcal{H}om(\overline C,L), \{\ell^\phi_k\}_{k\ge 1} )\le \nilsin (\mathcal{H}om(\overline C,L), \{\ell_k\}_{k\ge 1} )\le \nilsin L.$$
 \end{theorem}

 We also prove that $\Wh Z$, the rational Whitehead length  of a given space $Z$, is a  lower bound of the nilpotency order of any of its $L_\infty$ models and strict inequality may occur. Hence, denoting by $\nil Z$ the nilpotency order of the $L_\infty$ minimal model of  $Z$, we immediately obtain:

\begin{corollary}\label{nilpotencia} (i) If $X$ is a nilpotent finite complex, then:
$$ \max\{\Wh \map_f (X,Y),\Wh \map_f^*(X,Y)\} \le\nil Y.
$$

(ii) If $X$ is a nilpotent finite type complex, then:
$$ \max\{\Wh \widetilde{\map}_f (X,Y),\Wh \widetilde{\map}_f^*(X,Y)\} \le\nil {Y}.
$$
\end{corollary}

This generalizes \cite[6.2]{lupton2} since, for a coformal space $Z$ we have $\Wh Z=\nil Z$ (see  Proposition \ref{whitehead} for details).

 \section{$L_\infty$ rational homotopy theory}

 We shall be using known results on rational homotopy theory for
which \cite{FHT} is a standard reference. Here we simply recall that
a {\em Sullivan
algebra} $(\Lambda V, d)$ is a  \alg   where $\Lambda V$ is the free commutative algebra generated by the (positively graded) vector space $V$ which can be decomposed as the union of increasing graded subspaces $V(0)\subset V(1)\subset\cdots$ for which the differential $d|_{V(0)}=0$  and $d\bigl(V(k)\bigr)\subset\Lambda V(k-1)$. A Sullivan algebra is {\em minimal} if the differential is decomposable, i.e., $\im d\subset \Lambda^+V\cdot\Lambda^+V$. A Sullivan model (resp. Sullivan minimal model)  of a \alg $A$ is a quasi-isomorphism $(\Lambda V,d)\stackrel{\simeq}{\to}A$ from a Sullivan algebra (resp. Sullivan minimal algebra). A Sullivan model (resp. Sullivan minimal model) of a path connected space $X$ is a Sullivan model (resp. Sullivan minimal model) for the \alg, $A_{PL}(X)$, of rational polynomial forms on $X$.

Although in the framework of rational homotopy theory, most algebraic objects are concentrated in non negative degrees, we shall not make this assumption and every considered algebraic gadget, unless explicitly stated otherwise, will be $\bz$-graded.

 In what follows  graded coalgebras $C$ are  augmented  $\varepsilon \colon C\to \cu $ and
 co-augmented   $\cu \hookrightarrow C$ so that
 $\varepsilon (1)=1$ and $\Delta (1)=1\otimes 1$. We write ${\overline C}=\text{ker}\varepsilon$, so that $C=\cu \oplus
 {\overline C}$. Morphisms of coalgebras are assumed to preserve augmentations.

 As usual, the subspace of primitive elements is the kernel of the reduced comultiplication $\overline{\Delta }\colon {\overline C}\to
 {\overline C}\otimes {\overline C}$, $\overline{\Delta }c=\Delta
 c-c\otimes 1-1\otimes c$. Recall that $C$ is {\em primitively cogenerated } if ${\overline C}=\bigcup_n \text{ker}
 \overline{\Delta }^{(n)}$, being  $\overline{\Delta }^{(0)}=id_{{\overline C}}$,
 $\overline{\Delta }^{(1)}=\overline{\Delta }$ and  $\overline{\Delta }^{(n)}=(\overline{\Delta
 }\otimes id\otimes \cdots \otimes id)\circ \overline{\Delta
 }^{(n-1)}\colon {\overline C}\to {\overline C}\otimes \stackrel{n+1}{\cdots} \otimes{\overline C}$.
In the free commutative algebra $\Lambda V$ we consider the usual structure of cocommutative graded coalgebra
whose comultiplication  is the unique morphism of
graded algebras such that $\Delta (v)=v\otimes 1+1\otimes v$ with
$v\in V$.  This coalgebra is cofree in the following sense
 \cite[Lemma 22.1]{FHT}:  any linear map of degree zero from a  primitively cogenerated cocommutative graded
coalgebra $f\colon {\overline C}\to V$,
lifts to a unique morphism of graded coalgebras, $\varphi \colon
C\to \Lambda V$ such that $\xi \varphi _{|{\overline C}}=f$. Here,  $\xi\colon \Lambda^+V\to \Lambda^+V/\Lambda^{\ge 2}V\cong V$ denotes the projection.

We now present a quick overview  on $L_\infty$ algebras.
For a  compendium of known properties
we refer to \cite{lada1} or \cite{lada2}. For a more geometric and/or rational
homotopy theoretic flavored reference see \cite{getzler,kon}.

 \begin{definition}\label{linfinitodef}
       An {\em $L_\infty$ algebra} structure on a graded vector space $L$, often denoted by $(L, \{\ell_k\}_{k\ge 1})$, is a collection of linear maps, called brackets, $\ell_k$, $k\ge 1$, of degree $k-2$
       $$
       \ell_k=[\,,\dots,\,]\colon \otimes^kL\to L,
       $$
 which satisfy:

       (1) $\ell_k$ are graded skew-symmetric, i.e., for any $k$-permutation $\sigma$,
       $$
       [x_{\sigma(1)},\dots,x_{\sigma (k)}]=\sgn(\sigma)\varepsilon_\sigma[x_1,\dots,x_k],
       $$
       where $\varepsilon_\sigma$ is the sign given by the Koszul convention.

       (2) The following generalized Jacobi identities hold:
       $$
      \sum_{{\scriptstyle i+j=n+1}}\sum_{\sigma\in S(i,n-i)}{\textstyle{ \sgn(\sigma)\varepsilon_\sigma(-1)^{i(j-1)}\bigl[[x_{\sigma(1)},\dots,x_{\sigma(i)}],x_{\sigma(i+1)},\dots,x_{\sigma(n)}\bigr]=0.}}
       $$
       By $S(i,n-i)$ we denote the $(i,n-i)$ {\em shuffles} whose elements are permutations $\sigma$ such that $\sigma(1)<\dots<\sigma(i)$ and
$\sigma(i+1)<\dots<\sigma(n)$.
\end{definition}

Note that a \lie is an $L_\infty$ algebra in which $\ell_k=0$ for $k\ge 3$.

\begin{proposition}\label{propoto}  {\em \cite{lada1,lada2}} $L_\infty$ structures in $L$ are in one to one correspondence with codifferentials on the non unital, free cocommutative coalgebra $\Lambda^+ sL$ in which $s$ denotes suspension, i.e., $(sL)_k=L_{k-1}$.
\end{proposition}

In what follows, we often write $\mathcal{C}_\infty (L)=(\Lambda sL, \delta)$ for a given $L_\infty$ structure. Note that this is a generalization of the
 Cartan-Eilenberg-Chevalley complex
for a \lie.

\begin{definition}
       Given two $L_\infty$ algebras $L$ and $L'$, a {\em morphism of $L_\infty$
algebras} is a morphism of differential graded coalgebras
$f\colon \mathcal{C}_\infty (L)=(\Lambda sL,\delta)\to(\Lambda sL',\delta')=\mathcal{C}_\infty (L')$. Abusing notation, and whenever there is no ambiguity, we shall often write simply $f\colon L\to L'$.
\end{definition}

\begin{remark}\label{llavelinf}
Observe that, if $L$ is a finite type graded vector space, a
$L_\infty$ structure on $L$ induces an \alg structure
on $\Lambda (sL)^\sharp $. We shall denote
by $\cinf (L)$ this \alg structure. Explicitly
 $\cinf (L)=(\Lambda V,\dalg )$ is a \alg  in which $V$ and $sL$ are dual graded
vector spaces and $\dalg =\sum_{j \geq \pl } \dalg _j$ where $d_jV\subset\Lambda^jV$ and
$\langle \dalg _jv;sx_1\wedge \dots \wedge sx_j \rangle
=(-1)^\epsilon \langle v;s\left[ x_1,\dots ,x_j \right] \rangle $,
with $\epsilon=(-1)^{|v|+\sum_{k=1}^{j-1}(j-k)|x_k|} $.

 Again, this can be seen as a
generalization of the cochain algebra on a \lie  as $\langle d_1v;sx\rangle=(-1)^{|v|}\langle v;\partial x\rangle$ and $\langle d_2v;sx\wedge sy\rangle=(-1)^{|v|+|x|}\langle v;s[x,y]\rangle=(-1)^{|y|+1}\langle v;s[x,y]\rangle$.

Conversely, suppose $(\Lambda V,\dalg )$ is an arbitrary \alg
 of finite type. Then, an $L_\infty$ algebra structure in $s^{-1}V^\sharp$ is
 uniquely determined by the condition $(\Lambda V,\dalg )=\cinf
(L)$.

\end{remark}

 \begin{definition}\label{minimal}\cite[Def.4.7]{kon} An $L_\infty$ algebra $(L,\{\ell_k\}_{k\ge 1})$ is {\em minimal} if $\ell_1=0$.
   \end{definition}

 This definition is clearly compatible with the one arising from the minimality of Sullivan algebras  whenever $L$ is of finite type, non negatively graded and in which $L_0$ acts nilpotently in $L$ via $\ell_2$. Indeed, it is an easy exercise to show that, in this case, $L$ is minimal if and only if $\mathcal{C}^\infty(L)$ is a minimal Sullivan algebra.

Even if we deal with $L_\infty$ algebras which are not of finite type, nor non-negatively graded, an analogue machinery to that of Sullivan algebras can be developed. As in \cite{kon} we say that an $L_\infty$ algebra $(L,\{\ell_k\}_{k\ge 1})$ is {\em linear contractible} if $\ell_k=0$ for $k\ge 2$ and $H_*(L,\ell_1)=0$. As this property is not invariant under $L_\infty$ isomorphisms we say that $L$ is {\em contractible} if it is isomorphic as $L_\infty$ algebra to a linear contractible one. Then, one has:

\begin{proposition}\label{descomposicion}{\em \cite[4.9]{kon}} Any $L_\infty$ algebra  is $L_\infty$ isomorphic to the direct sum of a minimal and of a
linear contractible $L_\infty$ algebras.
\end{proposition}

Following \cite{kon},
we say that $f$ is a {\em quasi-isomorphism} if $f^{(1)}\colon (L,\ell_1)\stackrel{\simeq}{\to}(L',\ell'_1)$ is a quasi-isomorphism of differential graded vector spaces.
It is important to remark that if $f\colon L\to L'$ is an $L_\infty$ morphism  for which $f^{(1)}\colon L\stackrel{\cong}{\to}L'$ is an isomorphism of graded vector spaces, then $f$ is an isomorphism. This fact, and the proposition above, let us prove:

\begin{theorem}\label{unicidad} (i) If $f\colon\mathcal{C}_\infty(L)\longrightarrow \mathcal{C}_\infty(L')$ is a quasi-isomorphism of $L_\infty$ algebras there is another $L_\infty$ morphism $g\colon \mathcal{C}_\infty(L')\longrightarrow \mathcal{C}_\infty(L)$ for which the isomorphism $H_*(g^{(1)})\colon H_*(L',\ell'_1)\to H_*(L,\ell_1)$ is the inverse of $H_*(f^{(1)})$.

(ii) A quasi-isomorphism between minimal $L_\infty$ algebras is always an isomorphism.
\end{theorem}

Given $L_\infty$ algebras $L$ and $L'$,  we say that  $L$ is an {\em $L_\infty$ model} of $L'$ if $L$ and $L'$ are quasi-isomorphic. In view of Theorem \ref{unicidad}(i) above, this definition is appropriate as ``being quasi-isomorphic" is an equivalence relation among $L_\infty$ algebras. If $L$ is minimal we say that $L$ is the {\em minimal model} of $L'$. By Theorem \ref{unicidad}(ii), the minimal model of $L'$ is unique up to isomorphism of $L_\infty$ algebras.

\begin{definition}\label{linfinito}
(1)  An $L_\infty$ algebra $L$ (resp. $L_\infty$ minimal algebra $L$) of finite type is an {\em $L_\infty$ model} (resp. {\em $L_\infty$ minimal model}) of a finite type complex  $X$  if   $\cinfinito (L)$ is a Sullivan model (resp. Sullivan minimal model) of $X$.

(2)  An   $L_\infty$ minimal algebra $L$  of finite type is an   $L_\infty$ minimal model of a finite type complex  $X$  if   $\cinfinito (L)$ is a   Sullivan minimal model  of $X$.

(3)  An $L_\infty$ algebra  is an {\em $L_\infty$ model}   of a $1$-connected complex (not necessarily of finite type)  $X$ if it is a model  of the DGL $\lambda(X)$.
\end{definition}

Here $\lambda$ denotes the Quillen functor \cite{quillen} that associates to any 1-connected space $X$
a DGL (free as Lie algebra)
$\lambda(X)$, which determines an equivalence between the homotopy
categories of rational $1$-connected spaces  and that of reduced
($L_{\le0}=0$) DGL's over $\cu$.

\vspace{2mm}
\begin{remark}\label{uff}
Observe that, if  $X$ is $1$-connected and of finite type and $\mathbb L$ denotes the Quillen minimal model of $\lambda(X)$,
then $\calc^\infty(\mathbb L)$ is
 quasi-isomorphic to the Sullivan model of $X$ \cite{ma}. Thus, the above definitions are correct as both coincide in the intersection class of 1-connected, finite type, CW-complexes. Note also that, in this case, the $L_\infty$ minimal model of $X$ is $(L, \{\ell_k\}) $ in which $L\cong \pi_*(\Omega X)\otimes \mathbb Q$ and each $\ell_k$ is identified with the higher Whitehead product of order $k$.
\end{remark}

\begin{definition} The {\em  lower central series } of an $L_\infty $ algebra $L$ is, as in
the classical setting, defined inductively by
 $F^1L=L$
and, for $i>1$,
$$F^iL=\sum_{i_1+\cdots +i_k=i}\left[ F^{i_1}L,\dots , F^{i_k}L
\right].$$
$L$ is {\em nilpotent} if $F^iL=0$ for some $i\ge 1$.
\end{definition}

Following \cite{getzler}, given a nilpotent $L_\infty $ algebra $L$, the curvature for $\mc
\in L_{-1}$ is defined as
$$\mathcal{F}(\mc )=\partial \mc +\sum_{k =2}^\infty \frac{1}{k !}[ \mc ^{\wedge  k } ]\in L_{-2}$$
in which $\partial=\ell_1$ and $[ \mc ^{\wedge  k } ]=[\mc,\stackrel{k}{\ldots},\mc].$

\begin{definition}

The {\em Maurer-Cartan set} $\maca(L)$ of a nilpotent $L_\infty $ algebra
$L$ is the set of those $\mc \in L_{-1}$ of zero curvature, i.e, satisfying the
{\em Maurer-Cartan equation} $\mathcal{F}(\mc )=0$.
\end{definition}

Next, we recall how a Maurer-Cartan element of a given nilpotent $L_\infty $ algebra $L$ gives rise to another $L_\infty $ structure. In fact, for any $k\ge 1$ and
any element $\mc \in L_{-1}$, define
$$\ell_k^z(x_1,\dots,x_k)=[x_1,\dots ,x_k]_\mc =\sum_{j =0}^\infty \frac{1}{j!}[\mc ^{\wedge j }, x_1,\dots ,x_k]$$
in which $[\mc^{\wedge
j}, x_1,\dots ,x_k ]$ is an abbreviation for $[ \underbrace{\mc
,\dots ,\mc}_{j \text{ times}}, x_1,\dots ,x_k ]$. Observe that this is a perturbation of the original $L_\infty$ structure for which:

\begin{proposition}\label{relativa}{\em \cite[Prop.4.4]{getzler}}
If $\mc \in \text{\bf MC} (L)$, then $(L,\{\ell_i^z\})$ is  an $L_\infty $ algebra.
\end{proposition}

The study of Maurer-Cartan elements is at the basis of a recent and independent work of A. Berglund \cite{Berg}.

\vspace{2mm}\begin{example}
Let $X$ be a simply connected CW-complex with rational cohomology given by
$$H^*(X;\mathbb Q) \cong \Lambda (y, x_1, x_2, \ldots ) / (x_i^2-y^{2i})\,, \hspace{15mm} \vert x_i\vert = 2i\,, \vert y\vert = 2\,.$$
Since the sequence given by the $x_i^2-y^{2i}$ is a regular sequence, an $L_\infty$ model for $X$ is given by  the graded vector space generated by the elements $z, u_1, u_2, \ldots, y_1, y_2, \ldots $, with $\vert z\vert = 1, \vert u_i\vert = 4i-2$ and $\vert y_i\vert = 2i-1$. The only nonzero brackets are
$$-[ \underbrace{z
,\dots ,z}_{2i \text{ times}}] = [y_i,y_i]= u_i\,.$$
\end{example}

\section{$L_\infty$ models for the components of free and pointed mapping spaces}

We begin by applying the construction at the end of last section to a particular situation which will play an essential role
in the sequel. For it  observe that, if $L$ is an $L_\infty $ algebra and $C$ is a
cocommutative differential graded coalgebra primitively
cogenerated, then the complex of linear maps $Hom(C,L)$
(respectively $Hom({\overline C},L)$)  with
brackets

$$\left\lbrace \begin{array}{ l }
              \ell_1(f)=[f]=\ell_1\circ f+(-1)^{|f|+1}f\circ \delta \\
              \ell_k(f_1,\dots,f_k)=\left[f_1,\dots,f_k\right]=\ell_k \circ (f_1\otimes
              \cdots \otimes f_k) \circ \Delta^{(k-1)}  \ \ k\geq
              2,
           \end{array}
         \right.
            $$
(respectively with reduced codiagonal $\overline{\Delta
            }^{(k-1)}$) is an $L_\infty$ algebra.

On the other hand, each morphism  of differential graded coalgebras $\phi \in \text{CDGC}(C,\mathcal{C}_\infty (L))$  induces by restriction a Maurer-Cartan element in ${Hom}_{-1}(C,L)$ which we shall denote equally by $\phi$.
This gives a new $L_\infty $ structure on ${Hom}(C,L)$ with brackets
 $\ell^\phi_k=[\,,\dots ,]_\phi $.

\medskip

Next, we prove Theorems \ref{principallibre} and \ref{principalbasado}.
Let $\phi \colon (\Lambda V,d)\to B$  be a \alg
model of $f\colon X\to Y$ in which $(\Lambda V,
d)$ is the minimal Sullivan model of  $Y$ and $B$ is a connected finite dimensional \alg model of the finite CW-complex $X$.  Thus, if $L=s^{-1}V^\sharp$, then $\phi^\sharp\colon B^\sharp\to \mathcal{C}_\infty (L)$ is a CDGC morphism modeling $f$ which, by the preceding paragraphs, induces an $L_\infty$ structure on
 $Hom(B^\sharp ,L)$ with  brackets $[\,,
\dots ,]_\lineal $. Note that here $\overline{B^\sharp}=B^\sharp_+$.
 Hence,  Theorems \ref{principallibre} and \ref{principalbasado} are reformulated as:

 \begin{theorem} \label{modelotanre}
\mbox{}

\begin{enumerate}
\item[(i)]  $(\mathcal{H}om(B^\sharp,L), [\,,\dots ,]_\lineal )$  is
 an $L_\infty $
 model of {\em $\map_f(X,Y)$}.
\item[(ii)]
  $(\mathcal{H}om(\kerB ,L), [\,,\dots ,]_\lineal )$  is
 an $L_\infty $
 model of   {\em $\map_f^*(X,Y)$}.
\end{enumerate}
 \end{theorem}

\begin{remark}\label{casolie}
Observe that if $L$ is a \lie then we recover the DGL model given in
\cite[Theorem 10]{urtzi3} or \cite[Theorem 2]{urtzi4} for $\mathcal{H}om(\kerB ,L)$:
The bracket is given by
 $[f,g]= [\,,\,]\circ (f\otimes g)\circ\Delta$, and the differential by $D(f) = df + (-1)^{\vert f\vert} fd + [\phi, f]$.
\end{remark}

We first need to consider $(Der_\phi (\Lambda V,
B), \delta )$ the differential graded vector space of $\phi$-derivations where $Der_\phi
(\Lambda V,B)_n $ are  linear
maps $\theta \colon (\Lambda V)^*\to B^{*-n}$ for which $\theta (vw)=\theta
(v)\phi (w)+(-1)^{n|w|}\phi (v)\theta (w)$. The differential is
defined as usual $\delta \theta =d\circ \theta +(-1)^{n+1}\theta
\circ d$. Note that as graded vector spaces $Der_\phi (\Lambda V,
B)\cong Hom (V, B)$  via the identification
$\theta \mapsto \theta_{|V}$.  Then, we have an isomorphism:
$$\begin{aligned}
&\Theta \colon Der _\phi (\Lambda V, B)\stackrel{\cong
}{\longrightarrow} Hom (V\otimes B^\sharp, \cu ),\\
& \Theta
(\theta)(v\otimes \beta )=(-1)^{|\beta |(|v|+|\theta |)}\beta
(\theta (v)).
\end{aligned}$$

Consider the positive $\phi $-derivations defined by
$$\der_\phi (\Lambda V, B)_i= \left\lbrace
           \begin{array}{c l}
              Der_\phi (\Lambda V, B)_i& \text{for $i>1$},\\
              \mathcal{Z}Der_\phi (\Lambda V, B)_1 & \text{for $i=1$}.
           \end{array}
         \right.$$

Next, given $j\ge 2$ and  $\varphi _1,\dots ,\varphi _j\in \der _\phi (\Lambda V, B)$,
of degrees $p_1,\dots ,p_j$, we define their {\em bracket of length $j$}, $[\varphi_1,\dots
,\varphi_j]\in \der _\phi (\Lambda V, B)$, as in \cite[Def.14]{urtzi2} by
$$[\varphi_1{\scriptstyle,\dots,}\varphi_j](v)=(-1)^{p_1+\cdots+p_j-1}\sum ( \sum_{i_1,... ,i_j}
\varepsilon \phi (v_1{\scriptstyle\dots}\widehat{v}_{i_1}{\scriptstyle\dots}
\widehat{v}_{i_j}{\scriptstyle\dots}v_k)\varphi_1 (v_{i_1}){\scriptstyle\dots}
\varphi_j(v_{i_j}))$$ where $dv=\sum v_1\cdots v_k$ and $\varepsilon $
is the suitable sign given by the Koszul convention.

We may ``desuspend" these operations to define a set of linear maps $\{\ell_j\}_{j\ge 1}$, each of which of degree $j-2$, on
$s^{-1}\der_\phi (\Lambda V, B)$ as follows:

\smallskip

 For $j=1$,
 $$
 \ell_1\colon s^{-1}\der_\phi (\Lambda V, B)\to s^{-1}\der_\phi (\Lambda V, B)
 $$
  is the differential: $\ell_1(s^{-1}\varphi)=\delta s^{-1}\varphi
=-s^{-1}\delta \varphi$.

\smallskip

For $j\ge 2$ define
$$\ell_j\colon s^{-1}\der_\phi (\Lambda V, B)\otimes \cdots \otimes s^{-1}\der_\phi (\Lambda V,
B)\to s^{-1}\der_\phi (\Lambda V, B),$$
$$
\ell_j(s^{-1}\varphi_1,\dots  ,s^{-1}\varphi_j )=(-1)^\alpha
s^{-1}[\varphi _1,\dots ,\varphi_j ],
$$
where $\alpha
=\sum_{n=1}^{j-1}(j-n)|\varphi_n|$.

\smallskip

Note that these operations restrict to $s^{-1}\der_\phi (\Lambda V, \kerb )$.
Then we have:

\begin{lemma}\label{derivaciones}
$(s^{-1}\der_\phi (\Lambda V, B),\{\ell_j\}_{j\ge1})$ and $(s^{-1}\der_\phi (\Lambda V, B_+),\{\ell_j\}_{j\ge1})$ are $L_\infty $
models of $\text{\em map}_f(X,Y)$ and
$\text{\em map}_f^*(X,Y)$ respectively.
\end{lemma}

\begin{proof} By  Definition \ref{linfinito}(1) we have to prove that
 $\cinfinito s^{-1}\der_\phi (\Lambda V, B)$ is a Sullivan model of  $\text{map}_f(X,Y)$ (the pointed case is done similarly).
 For it let $(\Lambda W,d)$ be the Haefliger-Brown-Szczarba model of $\text{map}_f(X,Y)$ \cite{browncharba,urtzi2,Ha}. Then, in \cite[Theorem  5]{urtzi2} it is shown on one hand that  $W$ and $\der_\phi (\Lambda V, B)$ are dual vector spaces. In fact $W\subset V\otimes B^\sharp$ is precisely the dual of the image of the isomorphism $\Theta$ defined above restricted  to $\der_\phi (\Lambda V, B)$. In other words, $W^p=(V\otimes B^\sharp)^p$ for $p\ge 2$ and $W^1$ is a (possibly proper) subspace of $(V\otimes B^\sharp)^1$.  On the other hand, in \cite[Theorem  15]{urtzi2}, it is also established that, for any $j\ge 1$, $w\in W$ and $\varphi_1,\dots ,\varphi_j\in \der_\phi (\Lambda V, B)$,
 $$
 \langle \dalg _jw;\varphi_1\wedge \dots \wedge \varphi_j \rangle
=(-1)^{|w|+1} \Theta\bigl(s\ell_j(s^{-1}\varphi_1,\dots  ,s^{-1}\varphi_j ) \bigr)(w)\eqno{(1)}
$$
 With the notation of
 Remark \ref{llavelinf},  this is  $(-1)^\epsilon \langle w;s\ell_j(s^{-1}\varphi_1,\dots  ,s^{-1}\varphi_j )\rangle=(-1)^\epsilon \langle w;s[s^{-1}\varphi_1,\dots  ,s^{-1}\varphi_j ]\rangle$ and thus, via precisely this remark, the lemma is established.
\end{proof}

Note that, for the particular situation of $(\Lambda V, d=d_\pl+d_\pc )$,
Lemma \ref{derivaciones} is just \cite[Theorem 9]{urtzi3} in the case of
the trivial fibration.

\begin{proof}[Proof of Theorem \ref{modelotanre}]
 We give the proof for the free mapping space as the pointed case is done similarly. In view of Lemma \ref{derivaciones} it is enough to prove that $s^{-1}\der_\phi (\Lambda V, B
)$ and $\mathcal{H}om(B^\sharp ,L)$ are isomorphic $L_\infty $
algebras. First note that, as graded vector spaces, $s^{-1}Der_\phi
(\Lambda V, B )$ is isomorphic to $ Hom(B^\sharp ,L)$ via
$$
\begin{aligned}&s^{-1}Der_\phi (\Lambda V, B )_n=Der_\phi (\Lambda V, B
)_{n+1}\stackrel{\Theta }{\longrightarrow }Hom _{n+1}(V\otimes B^\sharp,\cu)\\
&\stackrel{\nabla }{\longleftarrow }Hom_{n+1}(B^\sharp , sL)=Hom_n(B^\sharp ,L)
\end{aligned}$$
in which $\nabla$ is the canonical isomorphism $\nabla \efe
(v\otimes \beta )=(-1)^{|v||\efe |}\langle v;\efe (\beta )\rangle$. For the pointed case it is easily
checked that this restricts to an isomorphism $s^{-1}\der_\phi
(\Lambda V, \kerb )\cong \mathcal{H}om(\kerB ,L)$.

We now show that this  isomorphism is compatible with the $\ell_j$'s for any $j\ge 1$.
For $j\ge 2$ this reduces  to the formal and straightforward verification that    given $\varphi_1,\dots,\varphi_j\in\der_\phi
(\Lambda V, B )$ and $f_1,\dots,f_j\in \mathcal{H}om(B^\sharp ,L)$ such that $\Theta(\varphi_i)=\nabla(f_i)$ for $1\le i\le j$, then
$$
\Theta [\varphi_1,\dots ,\varphi_j]=\nabla [\efe_1,\dots
,\efe_j]_\lineal .\eqno{(2)}
$$

We finish by showing that our isomorphism is also compatible with the differential $\ell_1$. Again, this reduces to show that given $\varphi\in\der_\phi
(\Lambda V, B )$ and $f\in \mathcal{H}om(B^\sharp ,L)$ such that $\Theta(\varphi)=\nabla(f)$,  then $\Theta \delta \varphi =\nabla
[\efe ]_\lineal $.  For it, let $v\otimes \beta\in W\subset V\otimes B^\sharp$. With the notation in Lemma \ref{derivaciones}, and via formula $(1)$, one has\footnote{For the sake of clarity,
 we shall omit signs in what follows and write just $\pm $. However a
careful application of the Koszul convention leads to proper sign
adjustments.}

$$
\Theta(\delta\varphi)(v\otimes\beta)=\pm\langle d_1(v\otimes\beta);\varphi\rangle=\Theta(\varphi)\bigl( d_1(v\otimes\beta)\bigr).
$$
However, in \cite[Lemma 6]{urtzi2}, it is shown that $d_1(v\otimes\beta)=\Gamma+(-1)^{|v|}v\otimes\delta\beta$ with $\Gamma\in V\otimes B^\sharp$ satisfying $\Theta(\varphi)(\Gamma)=\pm\beta\bigl(\varphi(dv)\bigr)$. Hence,
$$
\begin{aligned}
\Theta(\delta\varphi)(v\otimes\beta)&=\pm\beta\bigl(\varphi(dv)\bigr)\pm \Theta(\varphi)(v\otimes\delta\beta)\\
&=\pm\beta\bigl(\varphi(dv)\bigr)\pm \nabla(f)(v\otimes\delta\beta)=\pm\beta\bigl(\varphi(dv)\bigr)\pm \langle v;f(\delta\beta)\rangle.\quad(3)
\end{aligned}
$$

On the other hand $\nabla[f]_\lineal(v\otimes\beta)=\pm\langle v;[f]_\lineal(\beta)\rangle$. However,
$$
[f]_\lineal(\beta)=[f](\beta)+\sum_{\ell =1}^\infty \frac{1}{\ell !}[{\phi^\sharp}^{\wedge \ell }, f]=(-1)^{|f|} f(\delta\beta)+\sum_{\ell =1}^\infty \frac{1}{\ell !}[{\phi^\sharp}^{\wedge \ell }, f],
$$
and thus,
$$
\nabla[f]_\lineal(v\otimes\beta)=\pm\langle v;f(\delta\beta)\rangle\pm \sum_{\ell =1}^\infty \frac{1}{\ell !}\langle v;[{\phi^\sharp}^{\wedge \ell }, f](\beta)\rangle.\eqno{(4)}
$$
Therefore, to assure that $(3)=(4)$ it remains to prove that
$$\pm\beta\bigl(\varphi(dv)\bigr)=\pm \sum_{\ell =1}^\infty \frac{1}{\ell !}\langle v;[{{\phi^\sharp}}^{\wedge \ell }, f](\beta)\rangle.$$
This is the same  straightforward formal computation as in $(2)$ for $j=1$.

\end{proof}

We continue with:

\begin{proof}[Proof of Theorem \ref{casonofinito}]

We shall consider now  $f\colon X\to Y$   a pointed map between  rational CW-complexes of finite type ($X$ non necessarily finite!)
and denote by $f_n\colon X_n\to Y$ the restriction of $f$ to the $n$-skeleton of $X$.  Let $\mathcal{L}$ be a \lie model of $Y$ and choose, for each $n\ge 1$, an injective coalgebra model $i_n\colon C_n\hookrightarrow C_{n+1}$  of the inclusion $j_n\colon X_n\hookrightarrow X_{n+1}$ so that $C=\limite C_n$ is a finite type coalgebra model of $X$. Choose also $\phi_n\colon C_n\to \mathcal{C}(\mathcal{L})$ a \coa model of $f_n$ so that $\phi_{n+1}i_n=\phi_n$. In this case Theorem \ref{principallibre}, as observed in Remark \ref{casolie} above, asserts that, for each $n$, $\mathcal{H}om(C_n,\mathcal{L})$ with the differential given by $\partial f=\delta f+[\phi_n,f]$ and the $\ell_2$ bracket (note that $\ell_2=\ell_2^\phi$ as $\ell_k=0$ for $k\ge 3$) , is a \lie model of $\map_{f_n}(X_n,Y)$.  Moreover, $i_n^*\colon\mathcal{H}om(C_{n+1},\mathcal{L})\twoheadrightarrow \mathcal{H}om(C_n,\mathcal{L})$ is a surjective model of $j_n^*\colon \map_{f_{n+1}}(X_{n+1},Y)\to \map_{f_n}(X_n,Y)$. As $\limite \mathcal{H}om(C_n,\mathcal{L})=\mathcal{H}om(C,\mathcal{L})$ and $\limite \map_{f_n}(X_n,Y)=\map_{f}(X,Y)$, apply  \cite[Lemma 5]{urtzi4} to conclude that the universal cover of $\mathcal{H}om(C,\mathcal{L})$ is a \lie model of the universal cover of $\map_{f}(X,Y)$. At this point, it remains to show that if $L$ is an $L_\infty$ model of $Y$ then $\mathcal{H}om(C,\mathcal{L})$ and $(\mathcal{H}om(C,L), \{\ell^\phi_k\}_{k\ge 1} )$ are quasi-isomorphic as $L_\infty$ algebras. To this end apply the more general fact in the lemma below. Again, the pointed case is done mutatis mutandis replacing $C$ and $\Delta$ by $\overline C$ and $\overline\Delta$.
\end{proof}

Let $\gamma\colon \mathcal{C}_\infty(L)\to \mathcal{C}_\infty(L')$ be an $L_\infty$ morphism and let  $\phi\colon C\to \mathcal{C}_\infty(L)$ be a morphism of \coa, i.e., a Maurer-Cartan element of $Hom(C,L)$. We denote by
$$
\Gamma\colon (Hom(C,L), \{\ell^\phi_k\}_{k\ge 1} )\longrightarrow (Hom(C,L'), \{\ell^{\gamma\phi}_k\}_{k\ge 1} )
$$
the induced $L_\infty$ morphism for which
$$\Gamma^{(1)}\colon (Hom(C,L),\ell_1)\to  (Hom(C,L'),\ell_1^{\gamma\phi})$$ is precisely $Hom(C,\gamma^{(1)})$.

\begin{lemma}\label{lemafinal}
If $\gamma$ is a quasi-isomorphism of $L_\infty$ algebras, $C$ is of finite type and $C,L,L'$ are bounded below, then $\Gamma$
is also a quasi-isomorphism of $L_\infty$ algebras.
\end{lemma}
\begin{proof}
As $C$ if of finite type, a short computation shows that $(Hom(C,L), \ell^\phi_1 )$ is naturally isomorphic to $(C^\sharp\otimes L,d)$ in which $d(a\otimes b)=d a\otimes b+(-1)^{|a|}a\otimes\ell_1(b)+\Phi$ with $\Phi\in {C^\sharp}^{>|a|}\otimes L$. The same holds in $(Hom(C,L'), \ell^{\gamma\phi}_1)$. Then, filtering ${C^\sharp}\otimes L$  (resp. ${C^\sharp}\otimes L'$) by ${C^\sharp}^{\ge p}\otimes L$ (resp. ${C^\sharp}^{\ge p}\otimes L'$),
 $\Gamma^{(1)}=Hom(C,\gamma^{(1)})\colon (Hom(C,L), \ell^\phi_1 )\longrightarrow (Hom(C,L'), \ell^{\gamma\phi}_1)$ respects these filtrations and the resulting morphism of  spectral sequences at the $E_1$ level is $1_{C^\sharp}\otimes \gamma^{(1)}$ with split differentials. This is a quasi-isomorphism as $\gamma^{(1)}$ is. Our bounding assumption let us use comparison and thus, the lemma holds.
\end{proof}

\begin{example} We consider the injection $\varphi$ of $\mathbb CP^n $ into the first component of $ Y =\mathbb CP^{n+1}\#\,\, \mathbb CP^{n+1}$. A CDGC model for $\mathbb CP^n$ is given by the graded vector space $C$  generated by the elements  $1 = u_0, u_1, \ldots , u_n$, $\vert u_i\vert = 2i$, $\Delta (u_r) = \sum_{i+j=r} u_i\otimes u_j$.  An $L_\infty$ model for $Y$ is given by  the vector space $L$ with basis $a, b$ $c$, $v$, $\vert a\vert =\vert b\vert = 1$, $\vert c\vert = 2$, $\vert v\vert = 2n$ with nonzero brackets
$$[a,b]= c\,, \hspace{5mm} [\underbrace{a,\ldots ,a}_{n+1}] = [\underbrace{b,\ldots ,b}_{n+1}] = v\,.$$

A basis for the vector space $Hom_{\geq 0}(C,L)$ is given by the elements $f_0$, $ f_1$, $\ldots $, $f_n$, $g$, $h, k$ and $r$:
$$\begin{aligned} &f_i(u_i) = v,& f_i(u_j) = 0 \quad\mbox{for $i\neq j$},\quad& \vert f_i\vert = 2(n-i),\\
 &g(1) = a,& g(u_i) = 0\quad \mbox{ for $ i>0$},\quad&\vert g\vert = 1,\\
 &h (1) = b, & h (u_i) = 0 \quad \mbox{ for $ i>0$}, \quad&\vert h \vert = 1,\\
 &k(1) = c,&k(u_i)= 0\quad \mbox{ for $ i>0$}, \quad&\vert k\vert = 2,\\
& r(u_1) = c, & r(u_i) = 0 \quad \mbox{ for $ i\neq 1$},\quad& \vert r\vert = 0.\end{aligned}$$

We observe that $[h] = r$ and $[g] = \frac{1}{n!}f_n$. Therefore an $L_\infty$ model for  $\map_\varphi^*(\mathbb CP^n, Y)$ is given by the graded vector space generated by the elements $f_1, \ldots , f_n$ and $r$, all brackets being zero. On the other hand an $L_\infty$ model for $\map_\varphi (\mathbb CP^n, Y)$ is given by the vector space generated by the elements $f_0, f_1, \ldots , f_{n-1}$ and $k$, with all zero brackets.
  Both spaces have   the rational homotopy type of a product of odd dimensional spheres.
Denote by $Z$ the homotopy fiber of the map $f \colon  Y \to K(Z,2)\times K(Z,2)$ corresponding to the two generators of $H^2(Y)$. Then, $Z$ has the rational homotopy type of $S^3 \times S^{2n+1}$ and the injection $Z\to Y$ induces a rational homotopy equivalence
$\map^*(\mathbb CP^n, Y)\cong \map^*(\mathbb CP^n, Z)$.
\end{example}

As a final remark we observe that in general, the $L_\infty$ models obtained in this section are not minimal even if $L$ is, as the differential operator $\ell_1$ of these models may not be trivial. Nevertheless, there is a special case in which this occurs:

\begin{proposition}\label{modelominimal} If $X$ is formal, and for the constant map $X\to Y$, the $L_\infty$ models of Theorems \ref{principallibre}, \ref{principalbasado} and \ref{casonofinito} are minimal if $L$ is.
\end{proposition}

\begin{proof} For the constant map the $L_\infty$ model in Theorem \ref{principallibre} is  simply \break $(\mathcal{H}om(C,L),\{\ell_k\}_{k\ge 1})$, i.e., with no perturbation in the $k$-ary brackets. On the other hand, as $X$ is formal, we may choose $C$ to be $H^*(X;\cu)$ with zero differential. Therefore, $\ell_1=0$. Other cases are deduced analogously.
\end{proof}

\section{Rational nilpotency orders of mapping \break spaces}

 Let $L$ be a nilpotent $L_\infty$ algebra. We say that
 $L$ is of {\em  nilpotency order $i_0$} and write $\nilsin L=i_0$ if  $F^iL=0$ for $i>i_0$ and $F^{i_0}L\not=0$.  If $L$ is not nilpotent we write $\nilsin L=\infty$. This is not in general an $L_\infty$ invariant  as isomorphic nilpotent $L_\infty$ algebras can have different nilpotency orders. However, if  $L$ and $L'$ are isomorphic nilpotent minimal  $L_\infty$ algebras, then $\nilsin L=\nilsin L'$.

\begin{definition} Let $X$ be a CW-complex which is either  of finite type or simply connected. The {\em rational nilpotency order} of $X$, $\nil X$ is defined as $\nilsin L$, being $L$ the minimal $L_\infty$ model of $X$.
\end{definition}

This invariant can be related to the classical  {\em rational Whitehead length} of $X$, $\Wh X$, defined as the length of the longest non trivial iterated Whitehead bracket in $\pi_*(X_\cu)$. Indeed, we prove:

\begin{proposition}\label{whitehead} For any $L$, $L_\infty$ model of $X$, $\Wh X\le \nilsin L$. In particular, $\Wh X\le \nil X$ and equality holds if $X$ is coformal.
\end{proposition}

\begin{proof} It reduces to show that if $f\colon\mathcal{C}_\infty(L)\to \mathcal{C}_\infty(L')$ is a quasi-isomorphism of $L_\infty$ algebras in which
 $L'$ is a \lie, then $\nilsin L$ is greater than
 or equal that the usual nilpotency order of the \lie $H_*(L',\ell'_1)$. Let $\alpha,\beta\in H_*(L',\ell'_1)$ such that  $[\alpha,\beta]\not=0$. As $f$ is a
quasi-isomorphism, there are classes $\overline a,\overline b\in H_*(L,\ell_1)$ such that $H_*(f^{(1)})(\overline a)=\alpha$ and $H_*(f^{(1)})(\overline b)=\beta$. Then, as $a,b$ are cycles, the second equation satisfied by an $L_\infty$ morphism reads
$$
\ell'_1\bigl(f^{(2)}(a\otimes b)\bigr)+[f^{(1)}(a),f^{(1)}(b)]-f^{(1)}\bigl(\ell_2(a\otimes b)\bigr)=0.
$$
Passing to homology this equation becomes $[\alpha,\beta]-H_*(f^{(1)})\bigl(\overline{\ell_2(a\otimes b)}\bigr)=0$ and therefore
 $\ell_2(a\otimes b)$ is a non zero element in $F^2L$.
An inductive argument show then that any non zero iterated bracket of length $n$ in $H_*(L',\ell'_1)$ produces a non zero element in $F^nL$. Thus,
$$
 \nilsin H_*\bigl(\lambda(X)\bigr)=\Wh X\le \nilsin L
 $$
 and, in particular, $\Wh X\le \nil X$.
 Finally, if $X$ is coformal,  the homotopy Lie algebra of $X$ is precisely its minimal $L_\infty$ model.
 \end{proof}

\begin{proof}[Proof of Theorem \ref{lemanilpotencia}]
First, observe that if $(L,\{\ell_i\})$ and $(L,\{\ell_i^z\})$ are $L_\infty$ algebras
as in Proposition \ref{relativa}, then $\nilsin (L,\{\ell_i\})\ge \nilsin (L,\{\ell_i^z\})$. This is due to the fact that each $\ell_i^z$ is
 a perturbation of $\ell_i$, i.e., $[x_1,\dots ,x_k]_\mc=[x_1,\dots ,x_k]+\Phi$, $\Phi\in F^{>k}L$. Hence,
 $$\nilsin (\mathcal{H}om(C,L), \{\ell^\phi_k\}_{k\ge 1} )\le \nilsin (\mathcal{H}om(C,L), \{\ell_k\}_{k\ge 1} )$$
and
 $$\nilsin (\mathcal{H}om(\overline C,L), \{\ell^\phi_k\}_{k\ge 1} )\le \nilsin (\mathcal{H}om(\overline C,L), \{\ell_k\}_{k\ge 1} ).$$

 Thus,
 it remains to show that both $\nilsin (\mathcal{H}om(C,L), \{\ell_k\}_{k\ge 1} )$ and \break $\nilsin (\mathcal{H}om(\overline C,L), \{\ell_k\}_{k\ge 1} )$ are bounded above by $\nilsin L$. For it assume $\nilsin L=n_0$ and choose $f_1,\dots ,f_n\in \mathcal{H}om(C ,L)$ with $n>n_0$. Then, for any partition $i_1,\dots,i_k$ of $n$, consider the following commutative diagram:
$$\xymatrix{C \ar[rrrr]^{\bigl[[f_1,\dots f_{i_1}],\cdots ,[f_{{i_k}+1},\cdots ,f_n] \bigr]}
 \ar[d]_{{\Delta }^{(k-1)}}&&&&L\\
C \otimes \cdots \otimes C  \ar[rrrr]_{[f_1,\dots
,f_{i_1}]\otimes \cdots \otimes [f_{i_k+1},\dots ,
f_{n}]}\ar[d]_{{\Delta }^{(i_1-1)}\otimes \cdots \otimes
\Delta^{(i_k-1)}}&&&&L\otimes
\cdots \otimes L\ar[u]_{\ell_k }\\
C\otimes \cdots \otimes C \ar[rrrr]_{f_1\otimes
\cdots \otimes f_n}&&&&L\otimes \cdots \otimes
L\ar[u]_{\ell_{i_1}\otimes \cdots \otimes \ell_{i_k}} .}$$
As $\ell_k\circ (\ell_{i_1}\otimes \cdots \otimes \ell_{i_k})=0$, $\bigl[[f_1,\dots, f_{i_1}],\dots ,[f_{{i_k}+1},\dots ,f_n] \bigr]$ is
 also zero and $\nilsin \mathcal{H}om(C,L)\le n_0$. The same proof works for the ``pointed case''.
 \end{proof}

Observe that Corollary  \ref{nilpotencia} is immediately proved by applying together Proposition \ref{whitehead} and Theorem \ref{lemanilpotencia} to the $L_\infty$  models obtained in theorems \ref{principallibre}, \ref{principalbasado} and \ref{casonofinito}.  Remark also that, whenever the $L_\infty$ models of these theorems turn out to be minimal, the bounds in Corollary  \ref{nilpotencia} can be sharpened. For instance:

 \begin{corollary} If $X$ is a formal space and $c\colon X\to Y$ denotes the constant map, then
$$\nil map_c(X,Y)\le \nil Y$$  if $X$ is finite and
 $$ \nil \widetilde{map}_c(X,Y)\le \nil Y$$  if $X$ is of finite type. The same holds in the pointed case.
\end{corollary}

\begin{proof} We consider only the case in which $X$ is finite as the other is deduced analogously. By Proposition \ref{modelominimal}, if $L$ is the $L_\infty$ minimal model of $Y$, then
$(\mathcal{H}om(C,L),\{\ell_k\}_{k\ge 1})$ (resp. $(\mathcal{H}om(\overline C,L),\{\ell_k\}_{k\ge 1})$) is the $L_\infty$ minimal model of $\map_c(X,Y)$ (resp. $\map_c^*(X,Y)$). Thus, $\nil \map_c(X,Y)=\nilsin (\mathcal{H}om(C,L),\{\ell_k\}_{k\ge 1})$ (resp.
$\nil \map_c^*(X,Y)=\nilsin (\mathcal{H}om(\overline C,L),\{\ell_k\}_{k\ge 1})$), and Theorem \ref{lemanilpotencia} finishes the proof.
\end{proof}

\begin{example} As in \cite{lupton2}, consider the space $\map(S^3,Y)$ where $Y$ is the infinite dimensional complex  with Sullivan minimal model
$(\Lambda (x_2, y_3, z_3, t_7), d)$ in which (subscripts shall denote degrees henceforth) $x_2,y_3,z_3$ are cycles and
$dt_7=x_2y_3z_3$. In this case choose $C=\langle 1,\alpha \rangle$ with $\alpha$ a primitive cycle of degree $3$, and $L=\langle a_1,b_2,r_2,s_6\rangle$ the $L_\infty$ minimal model of $Y$ in which  the only non zero bracket is $\ell_3(a_1,b_2,r_2)=s_6$. Consider the map $f\colon S^3\to Y$ modeled by the Maurer-Cartan element $\phi\colon C\to L$, $\phi(\alpha)=b_2$.  By Theorem \ref{principallibre}, $\map_f(S^3,Y)$ and $\map_c(S^3,Y)$ are modeled by $(\mathcal{H}om(C,L),\{\ell_k^\phi\}_{k\ge1})$ and $(\mathcal{H}om(C,L),\{\ell_k\}_{k\ge1})$ respectively. Moreover both models are minimal as one observes that $\ell_k,\ell_k^\phi=0$ for $k=1$ and $k\ge 4$.
On the other hand if $f_1,g_2,h_2\in \mathcal{H}om(C,L)$  are given by $f_1(1)=a_1$, $g_2(1)=b_2$, $h_2(1)=r_2$ one checks that $\ell_3^\phi(f_1,g_2,h_2)(1)=s_6$ and therefore,
$$
\nil\map_f(S^3,Y)=\nil\map_c(S^3,Y) =\nil Y=3,
$$
according to Theorem \ref{lemanilpotencia}.

With respect to the rational Whitehead length, the situation is drastically different. Indeed, as the $L_\infty$ models above are minimal and of finite type, the Whitehead length of $\map_c(S^3,Y)$ and $\map_f(S^3,Y)$ are controlled by $\ell_2$ and $\ell_2^\phi$ respectively. While $\ell_2=0$ one checks that $\ell_2^\phi(f_1,g_2)(\alpha)=s_6$ and therefore,
$$
\Wh Y=\Wh \map_c(S^3,Y)=1<\Wh \map_f(S^3,Y)=2.
$$
\end{example}

\begin{example} Let $X$ be a space whose rational homotopy Lie algebra
 $L = \pi_*(\Omega X)\otimes \mathbb Q$ contains a free Lie algebra $\mathbb L(a_1,a_2)$ with $\vert a_1\vert,\vert a_2\vert >1$.
Then the rational homotopy Lie algebra  of the space $Z = \map^*(\mathbb CP^\infty, X)$ also contains a free Lie algebra with two generators.
 A coalgebra model $C$ for $\mathbb CP^\infty$ is given by the dual of the free algebra $\Lambda x$, with $\vert x\vert = 2$. That is, $C$ is   the vector space generated by  elements $x_i$ of degree $2i$, $i\ge 0$, and $\Delta (x_r) = \sum_{i+j=r} x_i\otimes x_j$.
An $L_\infty$ model for $X$ is given by $L$ with some higher order brackets. Denote  by $f_1$ and $f_2$ the elements of
$ Hom(C, L)$ defined by $f_i(x_1)=a_i$,  $f_i(x_j)=0$, for $j\not=1$.    Then $f_1$ and $f_2$ represent homotopy classes of $Z$. Now observe that
$$\bigl[f_{i_1}, [f_{i_2},[ \dots ,[f_{i_{r-1}}, f_{i_r}]\bigr]\dots\bigr] (x_{r+1}) = \bigl[a_{i_1}, [a_{i_1},[\dots ,[a_{i_{r-1}}, a_{i_r}]\bigr]\dots \bigr]\,.$$
This implies that $\mathbb L(f_1, f_2)\subset \pi_*(\Omega Z)\otimes \mathbb Q$.
\end{example}


\begin{thebibliography}{99}

\bibitem{Berg} A. Berglund, Rational homotopy theory of mapping spaces via Lie theory for $L_\infty$-algebras, Arxiv preprint, arXiv:1110.6145v1 (2011).

\bibitem{browncharba} E. H. Brown and R. H. Szczarba, On the rational homotopy type of function spaces, {\em Trans. Amer. Math. Soc.}, {\bf 349} (1997), 4931--4951.


\bibitem{urtzi2} U. Buijs and A. Murillo, The rational homotopy Lie algebra of function spaces,  {\em Comment. Math.
Helv.} {\bf 83} (2008), 723--739.

\bibitem{urtzi3} U. Buijs, Y. F\' elix and A. Murillo, Lie models
for the components of sections of a nilpotent fibration, {\em
Trans. Amer. Math. Soc.} {\bf 361}(10) (2009), 5601--5614.


\bibitem{urtzi4} U. Buijs, Y. F\'elix and A. Murillo, $L_\infty$ models of based mapping spaces,   {\em
J. of the Math. Soc.
of Japan}, {\bf 63} (2011), 503--524.

\bibitem{cs} M. Chas and D. Sullivan, String Topology, Arxiv preprint, arXiv:math/9911159v1  (1999).


\bibitem{FHT} Y. F\'elix, S. Halperin, J. Thomas, Rational
homotopy theory, {\em Springer GTM}, {\bf 205} (2000).

\bibitem{fuk} K. Fukaya, Deformation theory, homological algebra and  mirror symmetry. \textit{Geometry and physics of branes, Como 2001}, Ser. High Energy Phys. Cosmol. Gravit., IOP, Bristol, (2003), 121--209.

\bibitem{getzler} E. Getzler, Lie theory for nilpotent $L_\infty $-algebras, {\em Ann. of
Math.} {\bf 170} (2009), 271--301.

\bibitem{Ha} A. Haefliger, Rational Homotopy of the space of
sections of a nilpotent bundle, {\em Trans. Amer. Math. Soc.},
{\bf 273} (1982), 609--620.

\bibitem{hen} A. Henriques, Integrating $L_\infty$-algebras, {\em Compositio Mathematica}, {\bf 144} (2008), 1017--1045.

\bibitem{hub-kai} J. Huebschmann and T. Kadeishvili, Small models for  chain algebras. \textit{Math. Z.} \textbf{207}(2) (1991), 245--280.

\bibitem{kai} T. V. Kadeishvili, The algebraic structure in the  homology of an $A(\infty)$-algebra. \emph{Soobshch. Akad. Nauk Gruzin.  SSR} \textbf{108}(2) (1983), 249--252.



\bibitem{kon} M. Kontsevich, Deformation quantization of Poisson manifolds, {\em Lett. Math. Phys.}, {\bf 66}(3) (2003), 157--216.

\bibitem{kontsoib1} M. Kontsevich and Y. Soibelman, Deformations of  algebras over operads and Deligne's conjecture. {\em G. Dito and D.  Sternheimer (eds) Conf\'erence Mosh\'e Flato 1999, Vol. I (Dijon  1999), Kluwer
Acad. Publ., Dordrecht} (2000) 255--307.

\bibitem{kontsoib2} M. Kontsevich and Y. Soibelman, Homological mirror  symmetry and torus fibrations. \emph{Symplectic geometry and mirror  symmetry, Seoul 2000}, World Sci. Publ., River Edge, NJ,  (2001),  203--263.


\bibitem{lada1} T. Lada and M. Markl, Strongly homotopy Lie
algebras, {\em Comm. in Algebra}, {\bf 23}(6) (1995), 2147--2161.

\bibitem{lada2} T. Lada and J. Stasheff, Introduction to sh
algebras for physicists, {\em Int. J. Theor. Phys.}, {\bf 32}
(1993), 1087--1104.

\bibitem{bruno} J.-L. Loday and B. Vallette, Algebraic Operads (Version 0.99), in http://math.unice.fr/~brunov/Operads.pdf.

\bibitem{lupton2} G. Lupton and S. B. Smith, Whitehead products in
function spaces: Quillen model formulae, {\em Jour. Math. Soc.
of Japan}, {\bf 62} (2010), 1--33.

\bibitem{ma} M. Majewski, Rational homotopical models and
uniqueness, {\em Mem. Amer. Math. Soc.}, {\bf 682} (2000).

\bibitem{merk} S. A. Merkulov, Strong homotopy algebras of a  K\"{a}hler manifold.
\textit{Internat. Math. Res. Notices} {\bf 3} (1999), 153--164.

\bibitem{quillen} D. Quillen, Rational homotopy theory, {\em Ann.
of Math.}, {\bf 90} (1969), 205--295.


\bibitem{S-T} H. Scherer and D. Tanr\'e, Homotopie mod\'er\'ee et temp\'er\'ee avec les
coalg\`ebres. Applications aux espaces fonctionnels, {\em Arch.
Math.}, {\bf 59}(2) (1992), 130--145.


\bibitem{silveira}F. E. A. da Silveira, Rational homotopy theory of fibrations,
{\em Pacific Journal of Math.}, {\bf 113}(1) (1984), 1--34.

\bibitem{ss} M. Shlessinger and  J. Stasheff, The Lie algebra structure of tangent cohomology and deformation theory, {\em Jour. Pure Appl. Algebra}, {\bf 38} (1985), 313--322.


\end{thebibliography}
\end{document}